\documentclass [twoside,reqno,11pt]{amsart}

\usepackage[pagewise]{lineno}

\usepackage{psfrag}
\usepackage{hyperref}
\usepackage{amsfonts}
\usepackage{amssymb}
\usepackage{a4}
\usepackage{color}
\usepackage{mathtools}
\usepackage{comment}
\usepackage{geometry}
\usepackage{todonotes}

\newtheorem{thm}{Theorem}[section]

\newtheorem{prop}[thm]{Proposition}

\newtheorem{defn}[thm]{Definition}
\newtheorem{rem}[thm]{Remark}

\theoremstyle{definition}

\numberwithin{equation}{section}

\newcommand{\C}{\mathbb{C}}
\newcommand{\cD}{\mathcal{D}}
\newcommand{\cO}{\mathcal{O}}
\newcommand{\R}{\mathbb{R}}
\newcommand{\supp}{\operatorname{supp}}
\newcommand{\p}{\partial}

\newcommand{\LV}{\left|}
\newcommand{\RV}{\right|}

\begin{document}

\title[]{Partial Data Inverse Problems for Magnetic Schr\"odinger operators on conformally transversally anisotropic manifolds}

\author[Selim]{Salem Selim}

\address
        {Salem Selim, Department of Mathematics\\
         University of California, Irvine\\ 
         CA 92697-3875, USA }

\email{selimsa@uci.edu}

\author[Yan]{Lili Yan}

\address
        {Lili Yan, School of Mathematics\\
         University of Minnesota\\ 
         MN 55455, USA }

\email{lyan@umn.edu}

\maketitle

\begin{abstract} 
We study inverse boundary problems for the magnetic Schr\"odinger operator with H\"older continuous magnetic potentials and continuous electric potentials on a conformally transversally anisotropic Riemannian manifold of dimension $n \geq 3$ with connected boundary. A global uniqueness result is established for magnetic fields and electric potentials from the partial Cauchy data on the boundary of the manifold provided that the geodesic X-ray transform on the transversal manifold is injective.
\end{abstract}


\section{Introduction and statement of results}
Let $(M,g)$ be a  smooth compact oriented Riemannian manifold of dimension $n\ge 3$ with connected smooth boundary $\p M$. 
Let $d: C^\infty(M)\to C^\infty(M, T^*M)$ be the de Rham differential, and let $A\in C^\infty(M, T^*M)$ be a $1$-form with complex-valued $C^\infty$ coefficients. Let us introduce 
\[
d_A=d+i A: C^\infty(M)\to C^\infty(M, T^*M),
\]
and its formal $L^2$ adjoint of $d_A^*:C^\infty(M, T^*M)\to C^\infty(M)$ defined by $d_A^*=d^*-i \langle\overline{A}, \cdot\rangle_g$.

In this paper, we shall be concerned with inverse boundary problems for the magnetic Schr\"odinger operator with H\"older continuous magnetic potential $A\in C^{0,\varepsilon}(M, T^*M)$, $\varepsilon > 0$, and continuous electric potential $q\in C(M,\C)$ defined by
\begin{equation}
\label{eq_1_1}
\begin{aligned}
L_{g,A,q}u&=(d_{\overline{A}}^*d_A+q)u\\
&= -\Delta_g u +i d^*(Au)- i \langle A,du\rangle _g+ (\langle A, A\rangle_g+q)u,\quad u\in H^1(M^{int}),
\end{aligned}
\end{equation}
where $M^{int} = M\setminus\p M$ stands for the interior of $M$.

Let $u\in H^1(M^{int})$ be such that 
\[
L_{g,A,q}u=0 \quad \text{in} \quad \cD'(M^{int}).
\] 
Using a weak formulation, $(\p_\nu u+i\langle A, \nu\rangle u)|_{\p M}$ is well defined in $H^{-1/2}(\p M)$, see \cite[Section 1]{KK_2018}.  
Here and in what follows $\nu$ is the unit outer normal to the boundary of $M$. 

In this paper, our focus is to establish global uniqueness results for the magnetic potential $A$ and the electric potential $q$ from the knowledge of partial Cauchy data defined on a suitable open subset $\Gamma\subseteq \p M$ for solutions of the magnetic Schr\"odinger operator given by
\begin{align*}
    C^\Gamma_{g,A,q}=
    \{(u|_{\p M}, (\p_\nu u+i\langle A, \nu\rangle u)|_{\Gamma}): u\in H^1(M^{int}) \text{ such that } L_{g,A,q}u=0 \text{ in } \cD'(M^{int}) \}.
\end{align*}

A well-known feature of this problem is that there is gauge equivalence: one has
\begin{equation}
\label{eq_gauge}
C_{g,A,q } = C_{g,A+dp,q}
\end{equation}
for $p\in C^{1,\varepsilon}(M)$ such that $p|_{\p M}=0$, see \cite[Lemma 4.1]{KK_2018}.
Here $C_{g,A,q}$ is the full Cauchy data defined as follows:
\begin{align*}
    C_{g,A,q}=
    \{(u|_{\p M}, (\p_\nu u+i\langle A, \nu\rangle u)|_{\p M}): u\in H^1(M^{int}) \text{ such that } L_{g,A,q}u=0 \text{ in } \cD'(M^{int}) \}.
\end{align*}
Thus we may only hope to recover the magnetic field $dA$ and the electric potential $q$.
 
The study of the corresponding full data problem has been fruitful in the setting of $\R^n$ with $n\ge 3$. Following the fundamental works \cite{Sylvester_Uhlmann_1987} for Schr\"odinger operators i.e., $A = 0$, a uniqueness result for magnetic Schr\"odinger operators was obtained by Sun \cite{Sun_1993} for $A\in W^{2,\infty}$ under a smallness condition, and the smallness condition was later removed in \cite{NakSunUlm_1995} for smooth magnetic and electric potentials, and compactly supported $C^2$ magnetic and $L^\infty$ electric potentials. The regularity was extended to $A\in C^1$ in \cite{tolmasky1998exponentially}, to some less regular but small potentials in \cite{panchenko2002inverse}, and to Dini continuous magnetic potentials in \cite{salo2004inverse}. In particular, Krupchyk and Uhlmann \cite{Krup_Uhlmann_magnet} extended the uniqueness result for magnetic and electric potentials that are of class $L^\infty$. In three dimensions, Haberman \cite{haberman2018unique} improved the regularity to magnetic potentials  small in $W^{s,3}$ with $s>0$ and electric potentials in $W^{-1,3}$.

Going beyond the Euclidean setting, inverse boundary problems for magnetic Schr\"odinger operators were only studied in the case when $(M,g)$ is CTA  (conformally transversally anisotropic, see Definition \ref{def_CTA} below) and under the assumption that the geodesic $X$-ray transform on the transversal manifold is injective, see the fundamental works \cite{DKSaloU_2009} and \cite{DKuLS_2016} which initiated this study on CTA manifolds with simple transversal manifold, and on CTA manifolds with injective geodesic $X$-ray transform on the transversal manifold separately, see \cite{DKSalo_2013} for unbounded potentials. 
In the absence of $q$, this problem was studied in \cite{Cekic} for smooth magnetic potentials on CTA manifolds with injective geodesic $X$-ray transform on the transversal manifold. The regularity was improved in  \cite{KK_2018} for bounded magnetic and electric potentials when $(M,g)$ is CTA with a simple transversal manifold, and for a continuous magnetic potential and a bounded electric potential when $(M,g)$ is CTA with injective geodesic $X$-ray transform on the transversal manifold, see also \cite{KK_advection_2018}. We refer to the survey paper 
\cite{uhlmann2014inverse} for additional references for full data problems.

Turning our attention back to the partial data problem.
In the Euclidean setting, in the absence of a magnetic potential, the partial data result for Schr\"odinger operator is proved for $q\in L^\infty$ in \cite{kennig_sjostrand_uhlmann_partial_data} when $\Gamma$ is possibly very small, extended by \cite{DKSU_2007} to magnetic Schr\"odinger operator where both magnetic field $dA$ and the potential $q$ were uniquely determined. The regularity was relaxed to $A$ of  H\"older continuity, $q$ in $L^{\infty}$ in \cite{knudsen2006determining}. See \cite{j_chung_partial_2014,chung_partial_2014} for the case where both Dirichlet and Neumann data are measured on part of the boundary. 
On CTA manifolds with the absence of $A$, the partial data problem was studied for continuous $q$ in \cite{Kenig_Salo_APDE_2013}.
With the absence of $q$, this partial data problem was also studied in \cite{Cekic}.
Recently, a uniqueness result was proved in \cite{selim_partial_2022} for  $A\in W^{1,n}\cap L^\infty$ and $q\in L^n$ on CTA manifolds with a simple transversal manifold with $\Gamma$ being roughly half of the boundary, improving the uniqueness result obtained in \cite{Bhattacharyya_2018} for smooth $A$ and bounded $q$. 
We refer to the survey paper 
\cite{kenig2014recent} for a fuller account of the work done on partial data inverse problems.

To be on par with the best available full data result, one would like to establish a partial data result on CTA manifolds with injective geodesic $X$-ray transform on the transversal manifold. 

\begin{defn}
\label{def_CTA}
A compact Riemannian manifold $(M,g)$ of dimension $n\ge 3$ with boundary $\p M$ is called conformally transversally anisotropic (CTA) if $M\subset\subset \R\times M_0^{int}$  where $g= c(e \oplus g_0)$, $(\R,e)$ is the Euclidean real line, $(M_0,g_0)$ is a smooth compact $(n-1)$-dimensional manifold with smooth boundary, called the transversal manifold, and $c\in C^\infty(\R\times M_0)$ is a positive function. 
\end{defn}

Let us recall some definitions related to the geodesic X-ray transform following \cite{Guillarmou_2017}, \cite{DKSaloU_2009}. The geodesics on $M_0$ can be parametrized by points on the unit sphere bundle $SM_0=\{(x,\xi)\in TM_0: |\xi|=1\}$.  Let 
\[
\p_\pm SM_0=\{ (x,\xi)\in SM_0: x\in \p M_0, \pm \langle\xi,  \nu(x) \rangle>0\}
\] 
be the incoming ($-$) and outgoing ($+$) boundaries of $SM_0$. Here  $\nu$ is the unit outer normal vector field to $\p M_0$. Here and in what follows $\langle \cdot,\cdot\rangle$ is the duality between $T^*M_0$ and $TM_0$.

Let $(x,\xi)\in \p_-SM_0$ and $\gamma=\gamma_{x,\xi}(t)$ be the geodesic on $M_0$ such that $\gamma(0)=x$ and $\dot{\gamma}(0)=\xi$. Let us denote by $\tau(x,\xi)$ the first time when the geodesic  $\gamma$ exits $M_0$ with the convention that $\tau(x,\xi)=+\infty$ if the geodesic does not exit $M_0$. We define the incoming tail by
\[
\Gamma_-=\{(x,\xi)\in \p_-SM_0:\tau(x,\xi)=+\infty\}.
\]
When $f\in C(M_0,\C)$ and $\alpha\in C(M_0,T^*M_0)$ is a complex valued $1$-form, we define the geodesic X-ray transform on $(M_0,g_0)$  as follows:
\[
I(f,\alpha)(x,\xi)=\int_0^{\tau(x,\xi)} \big[ f(\gamma_{x,\xi}(t))+ \langle \alpha(\gamma_{x,\xi}(t)), \dot{\gamma}_{x,\xi}(t) \rangle \big]dt, \quad (x,\xi)\in \p_-SM_0\setminus\Gamma_-.
\]

A unit speed geodesic segment $\gamma=\gamma_{x,\xi}:[0,\tau(x,\xi)]\to M_0$, where $\tau(x,\xi)>0$, is called nontangential if $\gamma(0),\gamma(\tau(x,\xi))\in \p M_0$, $\dot{\gamma}(0),\dot{\gamma}(\tau(x,\xi))$ are nontangential vectors on $\p M_0$, and $\gamma(t)\in M_0^{int}$ for all $0<t<\tau(x,\xi)$.

\textbf{Assumption 1.} \textit{We assume that the geodesic $X$-ray transform on $(M_0,g_0)$ is injective in the sense that  if $I(f,\alpha)(x,\xi)=0$ for all $(x,\xi)\in \p_-SM_0\setminus\Gamma_-$  such that  $\gamma_{x,\xi}$ is a nontangential geodesic, then  $f=0$ and $\alpha=dp$ in $M_0$ for some $p\in C^1(M_0,\C)$ with $p|_{\p M_0}=0$}. 

Let $x=(x_1,x')$ be the local coordinates in $\R\times M_0$. Let $\varphi(x) = x_1$ be a limiting Carleman weight on $M$, see \cite{DKSaloU_2009}. We introduce the back side of $\p M$ as follows:
\begin{equation}\label{eq_backside}
B:=\{x\in \p M: \p_\nu \varphi (x)\ge 0\}.
\end{equation}
Our main result is the following:
\begin{thm}
\label{thm_main}
Let (M,g) be a CTA manifold of dimension $n\ge 3$ with a connected boundary such that Assumption 1 holds for the transversal manifold. Let $A^{(1)}, A^{(2)} \in C^{0,\varepsilon}(M,T^*M)$, $\varepsilon>0$, be complex-valued 1-forms, and $q^{(1)}, q^{(2)}\in C(M, \C)$. Let us assume further that $A^{(1)}|_{\p M}=A^{(2)}|_{\p M}$, $q^{(1)}|_{\p M}=q^{(2)}|_{\p M}$. Let $\Gamma \subset \p M$ be an open neighborhood of $B$. If $C^\Gamma_{g,A^{(1)}, q^{(1)}}=C^\Gamma_{g,A^{(2)}, q^{(2)}}$, then $dA^{(1)}=dA^{(2)}$ and $q^{(1)}=q^{(2)}$ in M.
\end{thm}
\begin{rem}
Theorem \ref{thm_main} can be viewed as an extension of \cite{KK_2018} from the full data case to the partial data case. Furthermore, Theorem \ref{thm_main} can be viewed as an improvement on \cite{KK_2018} in the sense that in \cite{KK_2018} only the magnetic field was recovered, while in our Theorem \ref{thm_main} both the magnetic field and the electric potential are recovered. From the perspective of geometric setting, Theorem \ref{thm_main} removes the simplicity assumptions on transversal manifolds in \cite{selim_partial_2022} and extends the unique determination of the magnetic field and potential to a larger class of CTA manifolds. 
\end{rem}

Let us proceed to discuss the main ideas in the proof of Theorem \ref{thm_main}. The main ingredients used to obtain the global uniqueness result are complex geometric optics (CGO) solutions for the magnetic Schr\"odinger operator constructed in \cite{KK_2018} based on Gaussian beam quasimodes, boundary Carleman estimates that controls the inaccessible part due to partial data, and an integral identity derived from \cite{selim_partial_2022}. 
Compared to \cite{selim_partial_2022}, the remainder terms in our CGO solutions decay slower as the semiclassical parameter approaches $0$. 
However, under the condition that $A^{(j)}, \in C^{0,\varepsilon}(M,T^*M)$, $\varepsilon>0$, $j = 1,2$, following the idea used by \cite{knudsen2006determining}, we may reduce the problem to the case when $d^*A^{(j)} = 0$, $j = 1,2$ with the help of Proposition \ref{prop_2_1}, see \cite[Lemma 2.2]{knudsen2006determining}. Therefore, the inaccessible part is still under control using the boundary Carleman estimates. 
\section{Proof of Theorem \ref{thm_main}}
\label{section_2}
Let $(M,g)$ be a CTA manifold so that $(M,g)\subset (\R\times M_0^{int}, c(e\oplus g_0))$, and let $A^{(1)}, A^{(2)} \in C^{0,\varepsilon}(M,T^*M)$, $\varepsilon > 0$, $q^{(1)}, q^{(2)}\in C(M, \C)$. We can assume that $d^*A^{(1)} = d^*A^{(2)}=0$ with the help of gauge equivalence \eqref{eq_gauge} and the following proposition. 
\begin{prop}
\label{prop_2_1}
If $A\in C^\varepsilon(M,T^*M), \varepsilon>0$, then there exists $p\in C^{1,\varepsilon}(M,\C)$ such that $d^*(A+dp)=0$ and $p|_{\p M}=0$.
\end{prop}
\begin{proof}
It suffices to choose $p$ such that $\Delta_g p=d^*A$ and $p|_{\p M}=0$, and this Dirichlet problem has a $C^{1,\varepsilon}$ solution by \cite[Theorem 8.34]{gilbarg1977elliptic}.
\end{proof}

Our starting point is the following integral identity from \cite[Proposition 4.4]{selim_partial_2022} which follows as a consequence of the equality $C^\Gamma_{g,A^{(1)}, q^{(1)}}=C^\Gamma_{g,A^{(2)}, q^{(2)}}$. By inspecting the proof of \cite[Proposition 4.4]{selim_partial_2022}, we get same integral identity for our regularity.
\begin{prop}
\label{prop_integral_identity}
Let $A^{(1)}, A^{(2)} \in C(M,T^*M)$, $d^*A^{(1)} = d^*A^{(2)}=0$, and $q^{(1)}, q^{(2)}\in C(M,\C)$. Assume that 
$C^\Gamma_{g,A^{(1)},q^{(1)}}=C^\Gamma_{g,A^{(2)},q^{(2)}}$. 
Then we have
\begin{equation}
\label{eq_integral_identity_1}
\begin{aligned}
&\int_M i\left\langle A^{(1)}-A^{(2)}, u_1d\overline{u_2}-\overline{u_2}du_1\right\rangle_g dV_g
+ \int_M\left(\langle A^{(1)}, A^{(1)}\rangle_g-\langle A^{(2)}, A^{(2)}\rangle_g +q^{(1)}-q^{(2)}\right)u_1\overline{u_2}dV_g\\
&=-\int_{\p M \setminus \Gamma}\p_\nu(m_2-u_1) \overline{u_2}dS_g+i\int_{\p M\setminus \Gamma}\left\langle A^{(1)}-A^{(2)}, \nu \right\rangle_gu_1\overline{u_2}dS_g,
\end{aligned}
\end{equation}
for $u_1, u_2\in H^1(M^{int})$ satisfying 
\begin{equation}
\label{eq_integral_identity_2}
L_{g, A^{(1)},q^{(1)}}u_1=0, \quad L_{g,\overline{A^{(2)}},\overline{q^{(2)}}}u_2=0 \quad \text{in} \quad \cD'(M^{int}),
\end{equation}
and $m_2\in H^1(M^{int})$ satisfying
\begin{equation}
\label{eq_integral_identity_3}
L_{g, A^{(2)},q^{(2)}}m_2=0 \quad \text{in} \quad \cD'(M^{int}), 
\end{equation}
such that 
\begin{equation}
\label{eq_integral_identity_4}
m_2|_{\p M}= u_1|_{\p M}, \quad \left(\p_\nu m_2+i\langle A^{(2)}, \nu \rangle_g m_2\right)|_\Gamma = \left(\p_\nu u_1+i\langle A^{(1)}, \nu \rangle_gu_1\right)|_\Gamma.
\end{equation}
\end{prop}

We shall also need the following complex geometric optics solutions based on Gaussian beam quasimodes for the semiclassical magnetic Schr\"odinger operator conjugated by a limiting Carleman weight constructed in \cite[Proposition 5.2 and Proposition 6.1]{KK_2018}.

Let $\gamma: [0,L]\to M_0$ be a unit speed non-tangential geodesic on $M_0$, and let $s=\frac{1}{h}+i\lambda$ with $\lambda\in\R$ being fixed. For all $h>0$ small enough, there exist $u_1,u_2\in H^1(M^{int})$ such that $L_{g,A^{(1)}, q^{(1)}}u_1=0$, $L_{g,\overline{A^{(2)}}, \overline{q^{(2)}}}u_2=0$
in $\cD'(M^{int})$ having the form 
\begin{equation}
\label{eq_CGO_1}
u_1= e^{-sx_1} c^{-\frac{n-2}{4}}(v_s+r_1),
\quad
u_2= e^{sx_1} c^{-\frac{n-2}{4}}(w_s+r_2),
\end{equation}
where $v_s, w_s\in C^\infty(M)$ are the Gaussian beam quasimodes such that 
\begin{equation}
\label{eq_CGO_2}
\begin{aligned}
&\|v_s\|_{H^1_{scl}(M^{int})} = \cO(1), \quad \|e^{sx_1}h^2L_{e\oplus g_0, A^{(1)}, q^{(1)}}e^{-sx_1}v_s\|_{H^{-1}_{scl}(M^{int})}=o(h),
\\
&\|w_s\|_{H^1_{scl}(M^{int})} = \cO(1), \quad \|e^{-sx_1}h^2L_{e\oplus g_0, \overline{A^{(2)}}, \overline{q^{(2)}}}e^{sx_1}w_s\|_{H^{-1}_{scl}(M^{int})}=o(h),
\end{aligned}
\end{equation}
and $r_j\in H^1(M^{int})$ are such that $\|r_j\|_{H^1_{scl}(M^{int})}=o(1)$ as $h\to 0$, $j = 1,2$.

Furthermore, for each $\psi\in C(M_0)$ and $x'_1\in \R$, we have
\begin{equation}
\label{eq_CGO_5}
\lim_{h\to 0} \int_{\{x'_1\}\times M_0} v_s \overline{w_s} \psi dV_{g_0}= \int_0^L e^{-2\lambda t} \eta(x_1,t)e^{\Phi^{(1)}(x'_1,t)+\overline{\Phi^{(2)} (x'_1,t)}} \psi(\gamma(t)) dt.
\end{equation}
 Here $\Phi^{(1)}, \Phi^{(2)}\in C(\R\times [0,L])$ satisfy the following transport equations,
\[
(\p_{x_1}-i\p_t) \Phi^{(1)}=-i A^{(1)}_1(x_1,\gamma(t))- A^{(1)}_t(x_1,\gamma(t)),
\]
\[
 (\p_{x_1}+i\p_t) \Phi^{(2)}=-i \overline{A^{(2)}_1(x_1,\gamma(t))}+\overline{A^{(2)}_t (x_1,\gamma(t))},
\]
where 
\[
A^{(j)}_t (x_1,\gamma(t))=\langle A^{(j)}(x_1,\gamma(t)), (0,\dot{\gamma}(t))\rangle, \quad j=1,2,
\]
with $\langle \cdot, \cdot \rangle$ being the duality between tangent and cotangent vectors, and $\eta\in C^\infty(\R\times [0,L])$ is such that $(\p_{x_1}-i\p_t) \eta=0$. 

Next, we shall test the integral identity \eqref{eq_integral_identity_1} against complex geometric optics solutions \eqref{eq_CGO_1}, multiply by $h$, and pass to the limit $h\to 0$. To that end, the following estimate for the right-hand side of \eqref{eq_integral_identity_1} is needed. 

\begin{prop}
\label{prop_rhsfirst}
Let $u_1,u_2,m_2$ be functions as described above.
Then we have 
\begin{equation}\label{eq_rhsfirst}
-h\int_{\p M \setminus \Gamma}\p_\nu(m_2-u_1) \overline{u_2}dS_g=o(1), \quad h\to 0.
\end{equation}
\end{prop}
\begin{proof}
Let us first recall that $\Gamma$ is an open neighborhood of $B$, given by \eqref{eq_backside}, we see that there exists $\varepsilon_0>0$ such that 
\[
B \subset  \widetilde{B}  \coloneqq \{ x\in \p M: \p_\nu\varphi(x)\ge -\varepsilon_0 \} \subset \Gamma.
\] 

By the CGO solution \eqref{eq_CGO_1} and the Cauchy-Schwartz inequality, we get
\begin{equation}\label{eq_rhsfirst_1}
\begin{aligned}
&\LV\int_{\p M \setminus \Gamma}\p_\nu(m_2-u_1) \overline{u_2}dS_g\RV
\\
&\le \frac{1}{\sqrt{\varepsilon_0}} \int_{\p M \setminus \widetilde B} \LV\sqrt{-\p_\nu \varphi} \p_\nu(m_2-u_1) \overline{u_2}\RV dS_g.
\\
&\le \frac{1}{\sqrt\varepsilon_0} \int_{\p M\setminus \widetilde B} \LV \sqrt{-\p_\nu \varphi} \p_\nu(m_2-u_1) e^{\overline{s}x_1}c^{-\frac{n-2}{4}}(\overline{w_s}+\overline{r_2}) \RV dS_g\\
&\le \cO(1) \| \sqrt{-\p_\nu \varphi} e^\frac{\varphi}{h}\p_\nu(m_2-u_1) \|_{L^2(\p M_-)} \left(\|w_s\|_{L^2(\p M \setminus \widetilde B)}+\|r_2\|_{L^2(\p M \setminus \widetilde B)}\right).
\end{aligned}
\end{equation}

To bound the first term in the last inequality in \eqref{eq_rhsfirst_1}, we shall recall the following boundary Carleman estimate for $L_{g,A,q}$ in \cite[Corollary 2.1]{selim_partial_2022}, and we note that, by inspecting the proof of \cite[Corollary 2.1]{selim_partial_2022}, the estimate is valid when $A\in C^{0,\varepsilon}(M, T^*M)$, $d^*A=0$, and $q\in C(M,\C)$. For $u\in H^2(M^{int})\cap H^1_0(M^{int})$ and $0<h\ll 1$, we have 
\begin{equation}\label{eq_boundary_carleman_est}
\begin{aligned}
h^{\frac{1}{2}}\|\sqrt{-\p_\nu\varphi} &e^{\varphi/h} \p_\nu u\|_{L^2(\p M_-)} + \|e^{\varphi/h}u\|_{H^1_{scl}(M^{int})}
\\
&\le \cO(h)\|e^{\varphi/h}L_{g,A,q}u\|^2_{L^2(M)}+\cO(h^{\frac{1}{2}})\|\sqrt{\p_\nu\varphi} e^{\varphi/h} \p_\nu u\|_{L^2(\p M_+)}.
\end{aligned}
\end{equation}
Here $\p M_\pm := \{x\in \p M: \pm\p_\nu\varphi(x)\ge 0 \}$
denote the front ($\p M_{-}$) and back ($\p M_{+}$) face of $\p M$, where $\varphi(x) = x_1$.

It follows from \eqref{eq_integral_identity_2}, \eqref{eq_integral_identity_3},  \eqref{eq_integral_identity_4}, and $d^*A^{(1)} = d^*A^{(2)}=0$ that $m_2 - u_1\in H^1_0(M^{int})$, $\Delta_g (m_2-u_1)\in L^2(M)$. Therefore, by the boundary elliptic regularity, we have $m_2 - u_1\in H^2(M^{int})$. 
Now apply the boundary Carleman estimate \eqref{eq_boundary_carleman_est} to 
$u=m_2-u_1$, $A=A^{(2)}$, $q=q^{(2)}$, we obtain   
\begin{equation}\label{eq_rhsfirst_3}
\begin{aligned}
&\| \sqrt{-\p_\nu \varphi} e^\frac{\varphi}{h}\p_\nu(m_2-u_1) \|_{L^2(\p M_-)}\\
&\le   \cO( \sqrt h) \|e^{\frac{\varphi}{h}} L_{g,A^{(2)},q^{(2)}} (m_2-u_1)\|_{L^2(M)} + \cO(1) \|\sqrt{\p_\nu \varphi} e^\frac{\varphi}{h} \p_\nu (m_2-u_1) \|_{L^2(\p M_+)}.
\end{aligned}
\end{equation}
Here the second summand vanishes by \eqref{eq_integral_identity_4} and $A^{(1)}|_{\Gamma}=A^{(2)}|_{\Gamma}$. In view of \eqref{eq_integral_identity_2}, \eqref{eq_integral_identity_3} and \eqref{eq_1_1}, we write 
\begin{equation}\label{eq_rhsfirst_4}
\begin{aligned}
e^{\frac{\varphi}{h}} L_{g,A^{(2)},q^{(2)}} (m_2-u_1)=&e^{\frac{\varphi}{h}} (L_{g,A^{(1)},q^{(1)}}-L_{g,A^{(2)},q^{(2)}}) u_1\\
=&e^{\frac{\varphi}{h}}\big( id^*(A^{(1)}-A^{(2)})u_1 - 2i\langle A^{(1)}-A^{(2)}, du_1 \rangle_g \\
&+( \langle A^{(1)}, A^{(1)} \rangle_g -\langle A^{(2)}, A^{(2)} \rangle_g+ q^{(1)}-q^{(2)})u_1 \big).
\end{aligned}
\end{equation}
Using $A^{(j)}\in C^{0,\varepsilon}(M,T^*M)$, $d^*A^{(j)}=0$, $q^{(j)}\in C(M,\C)$, $j=1,2$, we bound the first summand as follows: 
\begin{equation}\label{eq_rhsfirst_5}
\begin{aligned}
&\|e^{\frac{\varphi}{h}} L_{g,A^{(2)},q^{(2)}} (m_2-u_1)\|_{L^2(M)}\\
&\le \| e^{\frac{\varphi}{h}}2\langle A^{(1)}-A^{(2)}, du_1 \rangle_g \|_{L^2(M)}+\| e^{\frac{\varphi}{h}}( \langle A^{(1)}, A^{(1)} \rangle_g -\langle A^{(2)}, A^{(2)} \rangle_g+ q^{(1)}-q^{(2)})u_1  \|_{L^2(M)}\\
&\le  \cO(h^{-1})\|u_1\|_{H^1_{scl}(M^{int})}
=\cO(h^{-1}).
\end{aligned}
\end{equation}
In the last inequality, we used the fact that $\|u_1\|_{H^1_{scl}(M^{int})} = \cO(1)$, which is true by \eqref{eq_CGO_1}, \eqref{eq_CGO_2} and $\|r_1\|_{H^1_{scl}(M^{int})}=o(1), h\to 0$.

To bound the second term in the last inequality in \eqref{eq_rhsfirst_1}, we need the following semiclassical Sobolev trace estimate, see \cite[Chapter 6]{sjostrand_interior_boundary_semiclassical}:
\begin{equation}
\label{eq_trace}
\|v\|_{H^{1/2}_{scl}(\p M)}\le \cO (h^{-\frac{1}{2}})\|v\|_{H^{1}_{scl}(M^{int})}, \quad v\in H^1(M^{int}).
\end{equation} Using \eqref{eq_trace} together with $\|r_j\|_{H^1_{scl}(M^{int})}= o(1), j = 1,2 $, we obtain
\begin{equation}
\label{eq_reminder_trace}
\|r_j\|_{L^2(\p M)}\le \|r_j\|_{H^{1/2}_{scl}(\p M)} = o(h^{-\frac{1}{2}}),\quad h\to 0, \quad j = 1,2.
\end{equation}

To obtain the bounds 
\begin{equation}
\label{eq_quasimode}
\|v_s\|_{L^2(\p M\setminus \widetilde B)} = \mathcal{O}(1), \quad 
\|w_s\|_{L^2(\p M\setminus \widetilde B)} = \mathcal{O}(1),\quad h \to 0,
\end{equation}
we follow the same idea in the proof of \cite[Theorem 6.2]{Cekic} by noticing that we are taking $L^2$ norm over $\p M\setminus \widetilde B$ (and not over $\p M_-$). 
The fact that $\p M\setminus \widetilde B$ is a compact manifold with boundary of dimension $n-1$ and a projection argument are used here, see page $1826$ in \cite{Cekic}.

Combining the estimates\eqref{eq_rhsfirst_3}, \eqref{eq_rhsfirst_5}, \eqref{eq_reminder_trace}, and \eqref{eq_quasimode}, we obtain from \eqref{eq_rhsfirst_1} that 
\[
\left|\int_{\p M \setminus \Gamma}\p_\nu(m_2-u_1) \overline{u_2}dS_g\right| = o(h^{-1}), \quad h\to 0.
\]
This completes the proof of \eqref{eq_rhsfirst}.
\end{proof}


Noting that $\|u_j\|_{H^1_{scl}(M^{int})} = \mathcal{O}(1)$, $j= 1,2$, we have by the Cauchy-Schwartz inequality that
\[
\left|\int_M\left(\langle A^{(1)}, A^{(1)}\rangle_g-\langle A^{(2)}, A^{(2)}\rangle_g +q_1-q_2\right)u_1u_2dV_g\right|
= \cO(1),\quad h\to 0.
\]

The above estimate together with Propositions \ref{prop_rhsfirst} and $A^{(1)}|_{\p M} = A^{(2)}|_{\p M}$ implies from \eqref{eq_integral_identity_1} that, 
\begin{equation}\label{eq_lhs}
h\int_M \left\langle A^{(1)}-A^{(2)}, u_1d\overline{u_2}-\overline{u_2}du_1\right\rangle_g dV_g  = o(1),  \quad h\to 0.
\end{equation}

Estimate \eqref{eq_lhs} gives us exactly the same identity for $A^{(1)}-A^{(2)}$ as that in \cite[Section 7]{KK_2018}. Under the assumption that $A^{(1)}|_{\p M}=A^{(2)}|_{\p M}$, we may extend $\widetilde{A}:=A^{(1)} - A^{(2)} = 0$ by zero to the complement of $M$ in $\R\times M_0^{int}$, so that the extension $\widetilde{A}$ is continuous. 
Proceeding as in \cite[Section 7]{KK_2018} from \cite[Equation 7.1]{KK_2018} to \cite[Equation 7.9]{KK_2018} with the help of the concentrating property \eqref{eq_CGO_5}, we conclude from \eqref{eq_lhs} that 
\begin{equation}\label{eq_proof_1}
\int_0^L[f(\lambda, \gamma(t))-i\alpha(\lambda,\dot{\gamma}(t))]e^{-\lambda t}dt=0, 
\end{equation}
along any unit speed nontangential geodesic $\gamma:[0,L]\to M_0$ on $M_0$ and any $\lambda\in \R$. Here $f(\lambda,\cdot)\in C(M_0)$, $\alpha(\lambda, \cdot)\in C(M_0, T^*M)$ are as follows:
\begin{equation}\label{eq_proof_2}
\begin{aligned}
&f(\lambda,x')=\int_\R e^{-i\lambda x_1} \widetilde{A}_1 (x_1,x')dx_1,\quad x'\in M_0, \\
&\alpha(\lambda,x')=\sum_{j=2}^n \bigg(\int_\R e^{-i\lambda x_1 } \tilde{A}_j(x_1,x')dx_1 \bigg)dx_j.
\end{aligned}
\end{equation}

Arguing as in \cite[Section 7]{KK_2018}, see also \cite[Section 4]{Yan_2020}, \cite{Cekic}, we differentiate $f$ and $\alpha$ with respect to $\lambda$, and use the injectivity of the geodesic X-ray transform on functions and $1$-forms to conclude from \eqref{eq_proof_1} that there exist $p_l\in C^1(M_0)$, $p_l|_{\p M_0}=0$, such that 
\begin{equation}\label{eq_proof_3}
\p_\lambda^l f(0,x')+lp_{l-1}(x')=0, \quad \p_\lambda^l\alpha(0,x')=idp_l(x'),\quad l=0,1,2,\dots. 
\end{equation}

To proceed, we shall follow \cite[Section 4]{Yan_2020}, \cite[Section 5]{LO_2019}. Let 
\begin{equation}
\label{eq_proof_4}
\phi(x_1,x')=\int_{-a}^{x_1} \widetilde{A}_1(y_1,x')dy_1,
\end{equation}
where $\supp(\tilde{A}(\cdot,x'))\subset (-a,a)$.  It follows from \eqref{eq_proof_3}, \eqref{eq_proof_2} that 
\[
0=f(0,x')=\int_\R \widetilde A_1(y_1,x')dy_1,
\]
and therefore, $\phi$ has compact support in $x_1$. 

Thus, the Fourier transform of $\phi$ with respect to $x_1$, which we denote by $\widehat \phi(\lambda, x')$, is real analytic with respect to $\lambda$, and therefore, we have 
\begin{equation}
\label{eq_proof_5}
\widehat \phi(\lambda,x')=\sum_{k=0}^\infty \frac{\phi_k(x')}{k!}\lambda^k,
\end{equation}
where $\phi_k(x')=(\p_\lambda^k\widehat \phi)(0,x')$.  

It follows from \eqref{eq_proof_4} that 
\begin{equation}
\label{eq_proof_6}
\p_{x_1}\phi(x_1,x')=\widetilde A_1(x_1,x'),
\end{equation}
and therefore, taking the Fourier transform with respect to $x_1$, and using \eqref{eq_proof_2}, we obtain
\begin{equation}
\label{eq_proof_7}
i\lambda \widehat \phi(\lambda,x')=f(\lambda,x').
\end{equation}
Differentiating \eqref{eq_proof_7} $(l+1)$-times in $\lambda$, letting $\lambda=0$, and using \eqref{eq_proof_3},  we get 
\begin{equation}
\label{eq_proof_8}
\p_\lambda^l\widehat \phi(0,x')=ip_l(x'), \quad l=0,1,2,\dots.
\end{equation}

Substituting \eqref{eq_proof_8} into \eqref{eq_proof_5}, we obtain that 
\[
\widehat \phi(\lambda,x')=\sum_{k=0}^\infty \frac{ip_l(x')}{k!}\lambda^k,
\]
and taking the differential in $x'$ in the sense of distributions, and using \eqref{eq_proof_3}, \eqref{eq_proof_2}, we see that 
\begin{equation}
\label{eq_proof_9}
d_{x'}\widehat \phi(\lambda,x')=\sum_{k=0}^\infty \frac{i d p_l(x')}{k!}\lambda^k=\sum_{k=0}^\infty \frac{\p_{\lambda}^k\alpha(0,x')}{k!}\lambda^k=\alpha(\lambda,x')=\sum_{j=2}^n \widehat{\widetilde A_j}(\lambda,x')dx_j.
\end{equation}
Taking the inverse Fourier transform $\lambda\mapsto x_1$ in \eqref{eq_proof_9}, we get 
\begin{equation}
\label{eq_proof_10}
d_{x'} \phi(x_1,x')=\sum_{j=2}^n \widetilde A_j(x_1,x')dx_j.
\end{equation}

We also have from \eqref{eq_proof_6} that 
\begin{equation}
\label{eq_proof_11}
d_{x_1} \phi(x_1,x')=\widetilde A_1(x_1,x')dx_1.
\end{equation}

It follows from \eqref{eq_proof_11} and \eqref{eq_proof_10} that 
\begin{equation}
\label{eq_proof_12}
d\phi=\widetilde A.
\end{equation}
Since $\p M$ is connected and $d\phi|_{\p M}=\widetilde A|_{\p M} =0$, $\phi$ is a constant near $\p M$. Modifying $\phi$ by a constant, we may assume that $\phi=0$ on $\p M$.

By the natural obstruction \cite[Lemma 4.1]{KK_2018} and $\phi|_{\p M}=0$, we have $C_{g,A^{(2)},q^{(2)}}=C_{g,A^{(2)}+d\phi,q^{(2)}}$, and therefore
\begin{equation}\label{eq_proof_potential_1}
C^\Gamma_{g,A^{(2)}+d\phi,q^{(2)}}=C^\Gamma_{g,A^{(1)},q^{(2)}}.
\end{equation}
Then we may assume that $A^{(1)}=A^{(2)}$ and we will denote this 1-form by $A$. The integral identity \eqref{eq_integral_identity_1} now becomes 
\begin{equation}\label{eq_proof_potential_2}
\int_M(q^{(1)}-q^{(2)})u_1\overline{u_2}dV_g=-\int_{\p M \setminus \Gamma}\p_\nu(m_2-u_1) \overline{u_2}dS_g,
\end{equation}
for any $u_1, u_2, m_2 \in H^1(M^{int})$ described in Proposition \ref{prop_integral_identity} with $A^{(1)}=A^{(2)}=A$. 

We shall test the integral identity \eqref{eq_proof_potential_2} against complex geometric optics solutions to recover the electric potential. 

\begin{prop} \label{prop_electric_potential}
Let $u_1,u_2,m_2$ be functions as described above. Then we have 
\begin{equation}\label{eq_electric_potential_1}
-\int_{\p M \setminus \Gamma}\p_\nu(m_2-u_1) \overline{u_2}dS_g =o(1), \quad h\to 0.
\end{equation}
\end{prop}

\begin{proof}
In the same way as \eqref{eq_rhsfirst_1} and \eqref{eq_rhsfirst_3}, we get by boundary Carleman estimates that
\begin{equation}\label{eq_electric_potential_2}
\begin{aligned}
\left|\int_{\p M \setminus \Gamma}\p_\nu(m_2-u_1) \overline{u_2}dS_g\right| \le
&
\cO(\sqrt h) \|e^{\frac{\varphi}{h}} L_{g,A,q^{(2)}} (m_2-u_1)\|_{L^2(M)}
\left(\|w_s\|_{L^2(\p M\setminus \widetilde B)}+\|r_2\|_{L^2(\p M\setminus \widetilde B)}\right)
\end{aligned}
\end{equation}

Note that using \eqref{eq_integral_identity_2} and \eqref{eq_integral_identity_3}, we have
\begin{equation}\label{eq_electric_potential_3}
\begin{aligned}
L_{g,A,q^{(2)}}(m_2- u_1) &= L_{g,A,q^{(1)}} u_1 - L_{g,A,q^{(2)}} u_1 
= (q^{(1)}-q^{(2)})u_1.
\end{aligned}
\end{equation}
Thus, we get by \eqref{eq_electric_potential_3} that
\begin{equation}\label{eq_electric_potential_4}
\begin{aligned}
\|e^{\frac{\varphi}{h}} L_{g,A,q^{(2)}} (m_2-u_1)\|_{L^2(M)}\le
\mathcal{O}(1)\|q^{(2)}-q^{(1)}\|_{L^\infty(M)}\|u_1\|_{L^2(M)}.
\end{aligned}
\end{equation}
Estimate \eqref{eq_electric_potential_2}, together with \eqref{eq_electric_potential_4}, \eqref{eq_reminder_trace}, and \eqref{eq_quasimode} proves \eqref{eq_electric_potential_1}.
\end{proof}
Now combining \eqref{eq_proof_potential_2} with \eqref{eq_electric_potential_1}, we get 
\begin{equation}\label{eq_electric_potential_7}
    \int_M \left(q^{(1)}-q^{(2)}\right)u_1\overline{u_2}dV_g = o(1), \quad h \to 0.
\end{equation}
Using \eqref{eq_CGO_1}, \eqref{eq_CGO_2}, $\|r_j\|_{H^1_{scl}(M^{int})}=o(1)$, as $h\to 0$, and the Cauchy-Schwartz inequality, we obtain from \eqref{eq_electric_potential_7} that 
\begin{equation}\label{eq_electric_potential_8}
\int_M \left(q^{(1)}-q^{(2)}\right)e^{-2i\lambda x_1} c^{-\frac{n-2}{2}}v_s\overline{w_s}dV_g = o(1), \quad h \to 0.
\end{equation}

Under the assumption that $q^{(1)}|_{\p M}=q^{(2)}|_{\p M}$, we may extend $\widetilde{q}=q^{(1)} -q^{(2)} = 0$ by zero to the complement of $M$ in $\R\times M_0^{int}$, so that the extension $\widetilde{q}$ is continuous. 
Letting $h\to 0$, taking $\psi = \widetilde{q}c^{-\frac{n-2}{2}}$ in \eqref{eq_CGO_5}, and noting that $dV= c^{\frac{n}{2}}dV_{g_0}dt$, we have 
\begin{equation}\label{eq_electric_potential_9}
\int_\R\int_0^L e^{-2i\lambda x_1} e^{-2\lambda t} \eta(x_1,t)e^{\Phi^{(1)}(x'_1,t)+\overline{\Phi^{(2)} (x'_1,t)}} (\widetilde{q} c)(x_1,\gamma(t)) dtdx_1=0.
\end{equation}
We can take $\Phi^{(1)} = - \overline{\Phi^{(2)}}$ since $A =A^{(1)} = A^{(2)}$. Let us also take $\eta = 1$. Replacing $2\lambda$ by $
\lambda$, now \eqref{eq_electric_potential_9} reduces to 
\begin{equation}
\label{eq_electric_potential_10}
\int_{0}^{L}e^{-\lambda t} (\widehat{\widetilde{q} c})(\lambda,\gamma(t))dt = 0,
\end{equation}
for any $\lambda\in \R$ and any nontangential geodesic $\gamma$ in $M_0$, where 
\[(\widehat{\widetilde{q} c})(\lambda,\gamma(t)) = \int_{\R}e^{-i\lambda x_1}(\tilde{q} c)(x_1,\gamma(t))dx_1
\]
is analytic in $\lambda$ since it is the Fourier transform of $\widetilde{q} c$ in $x_1$ and $\supp{\widetilde{q} c}$ is compact.

Repeating similar arguments leading from \eqref{eq_proof_1} to \eqref{eq_proof_3} for $f(\lambda, x') = \widehat{\tilde{q} c}(\lambda,x')$ and \\ $\alpha(\lambda,x')=0$, we obtain 
\[
\p_\lambda^l (\widehat{\tilde{q} c})(0,\gamma(t))=0, \quad l=0,1,2,\dots. 
\]
By analyticity, we have $\widehat{\widetilde{q} c}=0$. Then using the injectivity of the Fourier transform, we recover $q^{(1)}=q^{(2)}$.

\section*{Acknowledgments}
The authors wish to express their sincere thanks to Hamid Hezari for the stimulating questions during Salem Selim's advancement, which motivated this paper. The authors gratefully acknowledge the many helpful suggestions of Katya Krupchyk during the preparation of this paper. The research of S.S. is partially supported by the National Science Foundation (DMS 2109199).

\bibliography{MagneticSchrodinger_partial_CTA.bib}
\bibliographystyle{plain}

\end{document}